\documentclass[11pt]{amsart}
\usepackage{amsfonts}

\usepackage[all,cmtip]{xy}
\usepackage{amsmath}
\usepackage{amsthm}
\usepackage{amssymb}
\usepackage{enumitem}       
\newtheorem{thm}{Theorem}[section]
\newtheorem{conj}{Conjecture}[section]
\newtheorem{prop}[thm]{Proposition}
\newtheorem*{definition*}         {Definition}
\newtheorem{lemma}[thm]{Lemma}

\newtheorem{cor}[thm]{Corollary}

\theoremstyle{remark}

\newcommand*{\Q}{\mathbb{Q}}
\newcommand*{\Hh}{\mathbb{H}}

\newcommand*{\Qa}{\overline{\mathbb{Q}}}
\newcommand*{\Z}{\mathbb{Z}}

\newcommand*{\GG}{\mathbb{G}}
\newcommand*{\R}{\mathbb{R}}

\newcommand*{\C}{\mathbb{C}}

\newcommand*{\PP}{\mathbb{P}}

\newcommand*{\ra}{\rightarrow}

\usepackage{amscd,amssymb,amsmath}
\usepackage{color}
\definecolor{purple}{rgb}{0.59, 0.44, 0.84}

\title{Independence of CM 
points in Elliptic Curves}
\author{Jonathan Pila and Jacob Tsimerman}
\begin{document}
\maketitle

\begin{abstract} We prove a result which describes, for each $n\ge 1$,
all linear dependencies among $n$ images in elliptic curves 
of special points in modular or Shimura curves under 
parameterizations  (or correspondences). Our result unifies and improves 
in certain aspects previous work of Rosen-Silverman--K\"uhne and Buium-Poonen.
\end{abstract}

\section{Introduction and main results}

Let $Y$ be a modular (or Shimura) curve, $E$ an elliptic curve over $\C$ 
and  $V\subset Y\times E$ an irreducible correspondence. If $(s,x)\in V$ we will call $x$
a {\it $V$-image\/} of $s$. We prove a result describing, for each $n\ge 1$, all linear dependencies 
in $E$ among the $V$-images of $n$ special points in $Y$.

An example of particular interest is when $V$ is the graph of a modular parameterization
$\phi: Y\rightarrow E$ and then the $V$-images of special points are known
as {\it CM points\/} or \emph{Heegner points} (though the latter term is usually taken to have 
some further assumptions). A number of results in the literature establish linear independence
of CM points under suitable hypotheses. After framing our result we compare it with previous results.

\begin{definition*}
{\rm With notation as above, and $n\ge 1$, let $\pi_{Y^n}, \pi_{E^n}$ be the projections
of $Y^n\times E^n$ onto the first and second factors, respectively.

(i) A {\it special graph\/} in $V^n$ is a component
$W\subset V^n\cap (S\times B)$, where $S\subset Y^n$ and $B\subset E^n$ are 
special subvarieties, such that $\pi_{Y^n}(W)=S$, $\pi_{E^n}(W)\subset B$.}

{\rm (ii) A special graph $W$ in $V^n$ is called {\it dependent\/} if $B$ 
(may be taken such that it) is a proper special subvariety.}

{\rm (iii) A special graph $W$ in $V^n$ is called {\it exemplary\/} if, setting $B$ to be the smallest special 
subvariety of $E^n$ with $\pi_{E^n}(W)\subset B$, there is no 
special graph $W'$ strictly larger than $W$ with $\pi_{E^n}(W')\subset B$.}
\end{definition*}

In particular when $V$ is the graph of a parameterization $\phi: Y\rightarrow E$, 
a special graph is simply the graph of the restriction of $\phi$ to a special 
subvariety $S\subset Y^n$.
The special subvarieties in $E^n$ are the cosets of abelian subvarieties by torsion
points (``torsion cosets''); the special subvarieties of $Y^n$ are described e.g. in \cite{PILA}.

Let $x_1,\ldots, x_n\in E$ be $V$-images of special points $s_1,\ldots, s_n\in Y$.
Write $s=(s_1,\ldots, s_n)\in Y^n, x=(x_1,\ldots, x_n)\in E^n$.
If $(s,x)\in W$ for some dependent special graph in $V^n$, then the points $x_1,\ldots, x_n\in E$
are linearly dependent in $E$. Note that, for us, {\it linear dependence in $E$\/} is always taken
to be over ${\rm End}(E)$. We have that ${\rm End}(E)=\Z$ unless $E$ has CM
(complex multiplication), in which case ${\rm End}(E)$ is an order in an imaginary
quadratic field.

Conversely, if $x_1,\ldots, x_n$ are linearly dependent in $E$ then $(s,x)$ is contained
in some exemplary dependent special graph.
Note that the unique non-dependent exemplary special graph is $V^n$ itself
as a component of $V^n\cap (Y^n\times E^n)$.

The following theorem thus gives a description of every linear dependence among
$V$-images of $n$ special points.

\begin{thm}\label{main}

Given $V\subset Y\times E$ as above with $Y$ a modular curve or a Shimura curve and $n\ge 1$, there are only finitely 
many exemplary special graphs in $V^n$.

\end{thm}

\noindent
{\bf Example.\/}
It is well known that $X_0(11)$ has the structure of an elliptic curve, so we may set $E={\rm Pic}^0(X_0(11))$. 
Consider the Atkin-Lehner involution $w$. Now on $E$, $w$ is a non-trivial
automorphism, so its graph must be an abelian subvariety. Set $\phi: X_0(11)\ra E$ to be 
the identification taking $(\infty)\ra 0$. Then $\phi(w(\infty)) = \phi(0)$ is a torsion point, and thus  
if we set $S\subset X_0(11)^2$ to be the graph of $w$, then $S$ is a special curve whose 
$\phi$-graph is exemplary. 

\medskip

A number of results in the literature assert linear independence properties of
the $V$-images of CM points. 
The fact that only finitely many $V$-images of special points can be 
torsion was proved in \cite{NEKOVARSCHAPPACHER}
for modular parameterisations and Heegner points
(generalized to certain Shimura curve parameterizations in \cite{KHARERAJAN})
and is equivalent to the assertion of Theorem 1.1 for $n=1$.
This also follows from the stronger results in \cite{BUIUMPOONEN},
and was reproved as a ``special point problem''
within the Zilber-Pink conjecture in \cite{PILA}.

We deduce some consequences of Theorem 1.1 and compare with
some further results in the literature. For $N\ge 1$ we let $X_N\subset Y\times Y$
be the locus of points $(s_1, s_2)$ such that there is a cyclic isogeny of degree $N$
between the corresponding elliptic curves (when $Y$ is a modular curve)
or abelian surfaces (when $Y$ is a Shimura curve).

\medskip
\noindent
{\bf Definition.} Let $D$ be a positive integer.  A set of special points 
$\{s_1,\ldots, s_n\}$ in $Y$ is called 
\emph{$D$-independent} 
if,  for each $i$, the discriminant $|\Delta(s_i)| > D$ and, 
for $i\ne j$, there is no relation $(s_i, s_j)\in X_N$ with $N\le D$.

\begin{cor}
For $n\ge 1$ there exists a positive integer
$D=D(Y, E, V, n)$ such that if $\{s_1,\ldots, s_n\}$ 
is $D$-independent then any $V$-images $x_1,\ldots, x_n$ of $s_1, \ldots, s_n$ 
are linearly independent in $E$.\ \qed
\end{cor}

\begin{proof}
For $s=(s_1,\ldots, s_n)$ to have a $V$-image which is dependent requires $s$ to lie
in one of finitely many proper special subvarieties $S_1,\ldots, S_k\subset Y^n$,
and, for each $i$, $s\in S_i$ requires either that some coordinate is equal to a fixed
special point, or some $(s_i, s_j)\in X_N$ for some $i\ne j$ and fixed $N$.
These are not possible if $s$ is $D$-independent for sufficiently large $D$.
\end{proof}

Corollary 1.2 improves a result of K\"uhne \cite{KUHNE}
(which in turn improved a result of Rosen-Silverman 
\cite{ROSENSILVERMAN})   
by getting independence even for CM points  corresponding to orders in the same CM field, if the
orders are ``sufficiently far apart'' (i.e. if the corresponding singular moduli are modularly independent up 
to suitable $D$); the previous results required the CM fields of the $s_i$ to be distinct.
The results in \cite{KUHNE, ROSENSILVERMAN} also exclude CM elliptic curves $E$, though see \cite{SAHINOGLU},
and all these results  restrict to modular parameterizations of $E/\Q$.
However, K\"uhne's result is effective,  whereas our result is not.


Let $\Sigma$ denote the set of $V$-images in $E$ of special points of $S$.

\begin{cor}
For $n\ge 1$ there exists a positive integer $N=N(Y, E, V, n)$
such that if $x_1,\ldots, x_N\in \Sigma$ are distinct then there is a linearly independent 
subset of $\{x_i\}$ of size at least $n$.
\end{cor}

\begin{proof}
Given $n$ we can find $N$ such that any set of $N$
distinct $V$-images of special points contains a subset of size $n$ for which
the corresponding special points are $D(Y, E, V, n)$-independent.
(And $N$ is effective given $D$.)
\end{proof}

\begin{cor}
Let $\Gamma$ be a finitely generated subgroup of $E$ of rank $r$.
Then $|\Gamma\cap \Sigma|\le N(Y, E, V, r+1)$.\ \qed
\end{cor}

This reproves a result of Buium-Poonen
(and generalizes to correspondences their result for maps from 
Shimura curves to elliptic curves) and in a uniform way:
the size of the intersection is bounded depending only on the rank of $\Gamma$.
However we cannot recover their ``Bogomolov''-type result.

In \S2 we show that Theorem 1.1 is a consequence of the Zilber-Pink conjecture (ZP).
The framing of ZP in terms of ``optimal subvarieties'' (as in \cite{HABEGGERPILA})
suggests the formulation of Theorem 1.1.

Our proof of Theorem 1.1 goes via point-counting on definable sets in o-minimal structures,
and utilizes a suitable Ax-Schanuel theorem, as have been employed in various earlier
work to tackle special cases of ZP, and in this respect follows in particular the approach in
\cite{PILATSIMERMAN} in studying ``CM-points'' for the multiplicative group.
As there, various issues arise from the fact
that we cannot prove the full Zilber-Pink statement for $V^n$.
But unlike in \cite{PILATSIMERMAN}, where we showed that no positive dimensional
dependent special graphs exist, we must here deal with this possibility,
which complicates the point-counting and the application of Ax-Schanuel, 
in view of our inability to affirm the full ZP. We must show that
we are able to restrict throughout to atypical intersections of a specific form. 

In effect, we must prove a very strong result of Andr\'e-Oort type: each proper 
special subvariety of $E^n$ has a pre-image in $Y^n$. This gives a countably 
infinite collection of subvarieties of $Y^n$ which is not contained in any algebraic family.
We must show that there are only finitely many special subvarieties  of $Y^n$
which are contained and maximal in any one of this countably infinite collection.

In the modular case we show that our results can be extended to include the Hecke orbits 
of a finite number of points in addition to special points.
The {\it Hecke orbit\/} of $u\in Y$ is $\{v\in Y: \exists N{\rm\ with\ } (u,v)\in X_N\}$.

\begin{definition*}

{\rm Let $Y$ be a modular curve and $U\subset Y$.}

{\rm (i) A {\it $U$-special point\/} of $Y$ is a point which is either
special or in the Hecke orbit of some $u\in U$. 

(ii) A {\it $U$-special point\/} in $Y^n$ is an
$n$-tuple of $U$-special points in $Y$.

(iii) A {\it $U$-special subvariety\/} of $Y^n$ is
a weakly special subvariety which contains a $U$-special point.}

\end{definition*}

Now we consider again an irreducible correspondence $V\subset Y\times E$.

\begin{definition*}

{\rm Let notation be as above.

(i) A {\it $U$-special graph\/} in $V^n$ is a component $W\subset V^n\cap (S\times B)$, where
$S\subset Y^n$ is $U$-special, $B\subset E^n$ is special, $\pi_{Y^n}(W)=S$, and $\pi_{E^n}(W)\subset B$. 

(ii) A $U$-special graph $W$ is {\it dependent\/} if $B$ (may be taken such that it) is a proper special subvariety.

(iii) A $U$-special graph $W$ is {\it exemplary\/} if,
setting $B$ to be the smallest special subvariety of $E^n$ with $\pi_{E^n}(W)\subset B$, there
is no $U$-special graph $W'$ strictly larger than $W$ with $\pi_{E^n}(W')\subset B$.}

\end{definition*}

\begin{thm}

Given $V\subset Y\times E$ as above with $Y$ a modular curve, $U\subset Y$ finite, and $n\ge 1$, there are only finitely 
many exemplary $U$-special graphs in $V^n$.

\end{thm}

One may deduce corollaries analogous to 1.2, 1.3, and 1.4 above. The last recovers
a result of Baldi (\cite{BALDI}, obtained via equidistribution) which is also a special case of results of 
Dill \cite{DILL, DILL2}, affirming a conjecture of Buium-Poonen \cite{BUIUMPOONEN2};
see the discussion in \cite{BALDI}. Baldi obtains a stronger ``Bogomolov''-type result, which we do not.
These results  are in the circle of the ``Andr\'e-Pink conjecture'', see \cite{PINK2}
and further references in \cite{BALDI}, though Theorem 1.5 is rather an ``unlikely intersection'' 
result in such contexts. Of course it too is subsumed under the general Zilber-Pink conjecture.

With existing arithmetic estimates Theorem 1.5 and its corollaries should generalize to
Shimura curves, with a suitable notion of Hecke orbit\footnote{There is an issue with abelian varieties that one could consider isogenies not necessarily respecting the polarization,
which complicates matters.}.

The structure of the paper is as follows. The Zilber-Pink setting is recalled in \S2.
The Ax-Schanuel statement and refinements we need are given in \S3. 
Some arithmetic estimates are collected in \S4. Theorems 1.1 and 1.5 are proved in \S5,
when everything is defined over a numberfield, and extended to $\C$ in \S6. 
In this paper, ``definable'' will mean ``definable in the o-minimal structure
$\R_{\rm an,\ exp}$''; for background on o-minimality and on 
$\R_{\rm an,\ exp}$ see \cite{PSDMJ}.

\medskip

\section{The Zilber-Pink setting}

We place Theorem 1.1 in the context of the Zilber-Pink conjecture (ZP) proposed
independently, in slightly different formulations, by Zilber \cite{ZILBERSUMS}, 
Bombieri-Masser-Zannier \cite{BMZANOMALOUS},  and Pink \cite{PINK}.

This concerns a {\it mixed Shimura variety\/} $M$ and its collection $\mathcal S$ of
{\it special  subvarieties}. One has also the larger collection of
{\it weakly special subvarieties\/}. For definitions see e.g. Gao \cite{GAO}. 
Let $Z\subset M$ be a subvariety.

For $S\in \mathcal S$,
a component $A\subset Z\cap S$ is {\it atypical\/} if
$$
\dim A> \dim Z+\dim S-\dim M.
$$
(The intersection is called {\it unlikely\/} if $\dim Z+\dim S-\dim M<0$.)
ZP predicts a description in finite terms of all ``atypical'' intersections
of $Z$ with special subvarieties $S\in\mathcal S$. 

For a subvariety $Z\subset M$ we let $\langle Z\rangle$ denote the smallest 
special subvariety of $M$ containing $Z$,
and by $\langle Z\rangle_{\rm ws}$ the smallest 
weakly special one. 

We define the \emph{defect} $\delta(Z)$ of $Z$
and the \emph{weakly special defect} $\delta_{\rm ws}(Z)$
by
$$
\delta(Z)= \dim\langle Z\rangle -\dim Z, 
\quad \delta_{\rm ws}(Z)= 
\dim\langle Z\rangle_{\rm ws} -\dim Z.
$$

\begin{definition*}
{\rm Let $Z\subset M$. 

(i) A subvariety 
$A\subset Z$ is called {\it optimal\/} if it is maximal for 
its defect as a subvariety of $Z$. That is,
if $A\subset B\subset Z$ and $\delta(B)\le\delta(A)$ 
then $B=A$.

(ii) A subvariety $A\subset Z$ is called 
{\it geodesic optimal\/}  if it is maximal for its 
weakly special defect as a subvariety of $Z$. }

\end{definition*}

The following is formally
equivalent to the strongest form of ZP, namely the analogue for a mixed Shimura
variety of the conjectures of Zilber and Bombieri-Masser-Zannier (for
semi-abelian varieties and ${\mathbb G}_{\rm m}$), 
as shown in \cite{HABEGGERPILA}. (The notion here called ``geodesic optimal''
was earlier introduced as ``cd-maximal'' in  a different context in
\cite{POIZAT} in the setting of ${\mathbb G}_{\rm m}$.)

\begin{conj} [ZP]
Let $Z\subset M$. Then $Z$ has only finitely many optimal subvarieties.
\end{conj}

The ambient variety $Y^n\times E^n$ is an example
of a {\it weakly special subvariety} of a 
{\it mixed Shimura variety\/} (it is {\it special} precisely
if $E$ has CM). Namely, let 
$$
\mathcal E\rightarrow Y
$$ 
be the universal family over $Y$ (of elliptic curves if $Y$ is a modular curve,
or of abelian surfaces if $Y$ is a Shimura curve). 
Then $\mathcal E$ is a mixed Shimura variety (see e.g. \cite{GAO}),
in which $Y$ can be identified with the zero-section.
If $E$ is isomorphic to the fibre over $s\in Y$ then it
may be identified with this fibre, which is weakly special. 
Correspondingly, 
$Y^n\times E^n$ may be identified with a weakly special subvariety of 
$\mathcal E^n\times \mathcal E^n$.

It is well-known, see e.g. Pink \cite{PINK}, that ZP implies a similar statement for its
weakly special subvarieties, whose ``special subvarieties''
are simply the intersections of it with special subvarieties of the ambient mixed Shimura variety.
There are corresponding notions of smallest special and weakly special subvariety containing
a given subvariety, defect and weakly special defect, and ZP can be expressed
in terms of the corresponding notion of ``optimal'' as in 2.1; in the sequel the notation
$\langle\cdot\rangle$ and defects will always be
with respect to the ambient variety $Y^n\times E^n$. In particular, we have:

\begin{definition*}
{\rm The {\it (weakly) special subvarieties} of $Y^n\times E^n$, in the above sense,
are products of (weakly) special  subvarieties in $Y^n$ and $E^n$, where the 
``special subvarieties'' of $E^n$ are its torsion cosets.}
\end{definition*}

It follows then that, for $Z\subset Y^n\times E^n$,
$$
\langle Z\rangle=\langle \pi_{Y^n}(Z)\rangle_{Y^n}
\times \langle \pi_{E^n}(Z)\rangle_{E^n}
$$
and likewise for $\langle Z\rangle_{\rm ws}$.

\bigbreak

Given $V\subset Y\times E$, we consider ZP
for $V^n\subset Y^n\times E^n$.
If $x\in E^n$ is a $V$-image of a special point $s\in Y^n$
and $x$ is dependent then $x\in B$ for some proper special subvariety of $E^n$. Then
$(s,x)\in V^n\cap (\{s\}\times B)$, and since $\dim (\{s\}\times B)+\dim V^n<2n$
this  shows that any dependent image of a special point is an ``unlikely'' or ``atypical''
intersection in the sense of the Zilber-Pink conjecture.

The following shows that exemplary  special
graphs are optimal subvarieties of $V^n$, and hence 
that Theorem 1.1 is a consequence of ZP.
However, we are not able to prove ZP for $V^n$ 
(once $n\ge 3$).

\begin{prop}
An exemplary special graph in $V^n$ is an optimal subvariety of $V^n$.
\end{prop}

\begin{proof}
Let $W\subset V^n\cap (S\times B)$ be an exemplary special graph with $\pi_{Y^n}(W)=S$
and $B=\langle\pi_{E^n}(W)\rangle$. Then $\dim W=\dim S$ and the smallest special 
subvariety of $Y^n\times E^n$ containing $W$ is 
$S\times B$. Hence the defect of $W$ is
$$
\delta(W)=\dim \langle W\rangle -\dim W = 
\dim S +\dim B -\dim W=\dim B.
$$
If $W$ were not optimal, it would be contained in 
some larger subvariety $W'\subset V^n$
of the same, or lower defect. Write 
$$
\langle W'\rangle =S'\times B'.
$$ 
Then
$B\subset B'$ and $\dim W'\le \dim S'$ and
$$
\delta(W')=\dim \langle W'\rangle -\dim W' = 
\dim S' +\dim B' -\dim W'.
$$
If $\delta(W')\le \delta(W)$ we must have $B'=B$ and 
$\dim W'=\dim S'$, which
would mean that $W'$ is a special graph in $V^n$ on $S'$, 
containing $W$, projecting into $B$.
But by the maximality of $W$ we have $W'=W$. 
\end{proof}

We will need the ``weak'' analogue of the above. 
A {\it weakly special graph in $V^n$\/} is a component $W\subset V^n\cap (S\times B)$
where $S, B$ are weakly special subvarieties. It is {\it exemplary\/}
if, taking $B=\langle \pi_{E^n}(W)\rangle_{\rm ws}$, there is no
weakly special graph $W'$ strictly larger than $W$ with $\pi_{E^n}(W')\subset B$.


\begin{prop}
An exemplary weakly special graph in $V^n$ is a
geodesic optimal subvariety of $V^n$.
\end{prop}

\begin{proof}
The same.
\end{proof}

The Ax-Schanuel theorem only detects weakly special 
subvarieties, and we thus need to show (as has
already been shown in several other settings,
including for all pure Shimura varieties by 
Daw-Ren \cite{DAWREN})
that optimal subvarieties are geodesic optimal.
For this we establish the ``defect condition''.

\medskip
\noindent
{\bf Definition.\/} A weakly special subvariety $X$ of a mixed Shimura variety has 
the {\it defect condition\/} if,
for $A\subset B\subset X$, we have
$$
\delta(B)-\delta_{\rm ws}(B)\le \delta(A)-\delta_{\rm ws}(A),
$$
the defects being with respect to the special and weakly special subvarieties of $X$.

\begin{prop}
Let $S$ be a pure Shimura variety and $T$ an abelian 
variety.
Then $S\times T$ has the defect condition.
\end{prop}

\begin{proof}
For an abelian variety (as well as for $\mathbb G_{\rm m}^n$
and products of modular curves) the defect condition is established
in \cite{HABEGGERPILA}, Proposition 4.3,
and for a general pure Shimura variety  in \cite{DAWREN}, 4.4.
Since the (weakly) special subvarieties of $S\times T$ are products of 
(weakly) special subvarieties of the factors, we have
$$
\langle A\rangle=\langle \pi_S(A)\rangle_{S}\times \langle \pi_T(A)\rangle_T
$$
so that
$$
\delta(A)-\delta_{\rm ws}(A)=\delta(\pi_S(A))-\delta_{\rm ws}(\pi_S(A))+\delta(\pi_T(A))-\delta_{\rm ws}(\pi_T(A)),
$$
and likewise for $B$, and the defect condition for $S\times T$ follows from the
defect conditions in $S$ and $T$ by addition.
\end{proof}

It is conjectured in \cite{HABEGGERPILA}
that the defect condition holds in all mixed Shimura varieties. Presumably a proof can't be too far 
from  the above, as the weakly specials are ``nearly'' products, i.e. they are flat over a pure special.

\begin{prop}
An optimal subvariety is geodesic optimal.
\end{prop}

\begin{proof}
This follows formally once one has the defect condition, as in \cite{HABEGGERPILA}.
\end{proof}

\section{Ax-Schanuel}

The {\it Ax-Schanuel property} for the uniformization map 
$$
u_M: D\rightarrow M
$$ 
realizing a mixed Shimura variety $M$ as a quotient of a suitable Hermitian symmetric 
domain $D$ by a discrete arithmetic group $\Gamma$  is a functional transcendence statement 
for $u_M$ analogous to the classical
Ax-Schanuel theorem
for the exponential function $\exp: \C\rightarrow \C^{\times}$.
For discussion and proof of such results see \cite{MPT, GAO}.
Such a result implies a corresponding statement for each weakly special subvariety $X\subset M$,
uniformized by an irreducible component of $u_M^{-1}(X)$.

The Ax-Schanuel result we need is for (all the cartesian powers of) the uniformization
$$
u: \Hh\times\C\rightarrow Y\times E.
$$
We will use the same notation
$u: \Hh^n\times\C^n\rightarrow Y^n\times E^n$ for cartesian powers.
Since this is the uniformization 
corresponding to a weakly special subvariety of $\mathcal E^{2n}$, the result follows
from the Ax-Schanuel statement for the uniformization
$$
\Hh^{2n}\times\C^{2n} \rightarrow \mathcal E^{2n}
$$
and since $Y^{2n}$, the ``pure'' Shimura variety underlying ${\mathcal E}^{2n}$, 
is a special subvariety of $\mathcal A_{g}$, the Siegel modular variety of
principally polaried abelian varieties, when $g\ge 2n$ the required
Ax-Schanuel follows from the corresponding statement for
the universal family $\mathcal{X}_{g}$ of abelian varieties
over $\mathcal{A}_{g}$, namely the Ax-Schanuel theorem for 
the uniformization
$$
\Hh_g\times \C^g\rightarrow \mathcal{X}_g.
$$
This theorem is due to Gao \cite{GAO}, Theorem 1.1, extending, for $\mathcal{A}_g$, 
the result for a general pure Shimura variety in \cite{MPT}, Theorem 1.1.

We will (as usual in ZP applications) use only the 
``two-sorted'' form, which we now state for the uniformization 
$u: \Hh^n\times\C^n\rightarrow Y^n\times E^n$, after noting the following convention.

Strictly speaking $\Hh^n$ has 
no ``algebraic subvarieties''; by an \emph{algebraic subvariety\/}
of $U$, where $U\subset \Hh^n\times\C^n$ is a weakly special subvariety, 
we will mean an irreducible analytic component of the intersection of $U$
with an algebraic subvariety (in the usual sense)
of the ambient $\C^n\times\C^n$.

\begin{thm}
Let $U'$ be a weakly special subvariety of $\Hh^n\times \C^n$ with image 
$u(U')=X'$ a weakly special subvariety of $Y^n\times E^n$.
Let $Z\subset X'$, $A\subset U'$ be algebraic varieties, 
and $\Omega$ an irreducible analytic component
of $A\cap u^{-1}(Z)$. Then
$$
\dim\Omega=\dim Z +\dim A-\dim X'
$$
unless $\Omega$ is contained in a proper weakly special subvariety of $U'$.\ \qed
\end{thm}

As in \cite{HABEGGERPILA, DAWREN}, this can be reformulated in terms of a suitable notion of
``optimality'', for which we adopt the terminology used by Daw-Ren \cite{DAWREN}, \S5.7-5.9,
to distinguish it from ``optimality'' as above in \S2.

\medskip
\noindent
{\bf Definition.\/} Let $Z\subset Y^n\times E^n$ be a subvariety.

(i) An \emph{intersection component} of $u^{-1}(Z)$ is an irreducible analytic component
of the intersection of $u^{-1}(Z)$ with an algebraic subvariety of $\Hh^n\times \C^n$.

(ii) If $A$ is an intersection component of $u^{-1}(Z)$ with Zariski closure $\langle A\rangle_{\rm Zar}$
we define its \emph{Zariski defect} to be
$$
\delta_{\rm Zar}(A)=\dim \langle A\rangle_{\rm Zar} -\dim A.
$$

(iii) An intersection component $A$ of $u^{-1}(Z)$ is called \emph{Zariski optimal\/} if one 
cannot find a larger intersection component  of $u^{-1}(Z)$ which does not 
increase the Zariski defect.

(iv)  An intersection component $A$ of $u^{-1}(Z)$ is called
\emph{geodesic} if $A$ is a component of $u^{-1}(Z)\cap \langle A\rangle_{\rm Zar}$
and $\langle A\rangle_{\rm Zar}$ is weakly special.

\medskip

\begin{prop} Let $Z\subset Y^n\times E^n$ be a subvariety.
A Zariski optimal component  of $u^{-1}(Z)$ is geodesic.
\end{prop}

\begin{proof}
The equivalence of 3.1 and 3.2 is purely formal and the proof is carried out in
\cite{HABEGGERPILA}, below 5.12.
\end{proof}

\noindent
{\bf Definition.\/} A {\it M\"obius subvariety\/} of $\Hh^n$ is an algebraic subvariety
defined by setting some coordinates constant, and relating some other pairs
of coordinates by elements of ${\rm SL}_2(\R)$.

\medskip

We let $F$ denote a standard fundamental domain 
for the uniformization of $Y\times E$. The uniformization map restricted to $F$
is definable (in this case by results of Peterzil-Starchenko \cite{PSDMJ}),
and the M\"obius subvarieties of $\Hh^n$ form a definable family.

This means that if we consider the definable family of subvarieties of 
$\Hh^n\times\C^n$ comprising all products of ``M\"obius suvarieties'' of $\Hh^n$
and linear subvarieties of $\C^n$, and define the set of
Zariski optimal ones by the difference of their dimension and
dimension of intersection with $u^{-1}(V)$,
just among these which go through $F$, we will get the slopes
(up to ${\rm SL}_2(\Z)$ and $\Lambda$) of all geodesic optimal components.
This then implies the finiteness of such slopes in $Y^n\times E^n$,
and any geodesic optimal component of $V^n$ will have some
pre-image component going through $F$.

We want the corresponding finiteness for the particular type of components
we consider. Namely, if $W$ is a dependent special graph, we consider
a component $U$ of its pre-image in $\Hh^n\times \C^n$. It is a component
of the intersection of $u^{-1}(V^n)$ with suitable pre-image $M\times L$  of 
$\langle W\rangle =S\times B$, and is thus a geodesic component
which projects onto $M$ and thus has $\dim U=\dim M$.

We need to observe that, if Zariski optimal, such a component comes from a 
maximal dependent (weakly) special image, i.e. something of the same form.
%
%
%
%
In fact we need something further along these lines in the proof of 1.1, 
in order to get from ``something positive-dimensional algebraic'' to a component
of the right form.

\begin{prop}\label{zardef}
Let $U$ be of the following type: it is a component of $A\times L$ intersecting
$u^{-1}(V^n)$, where $A$ is algebraic, and $L$ is linear which projects onto $A$.

If $U$ is maximal of this type for the given $L$ then $L$ (and $A$)
are weakly special and $U$ is Zariski optimal.

\end{prop}

\begin{proof}
We have $\dim U=\dim A$ and so
$$
\delta_{\rm Zar}(U) \le \dim L.
$$
Suppose that $U\subset U'$, with $U'$ Zariski optimal, 
and hence geodesic optimal, with $U'$ a component of 
the intersection of $u^{-1}(V^n)$ with weakly special $A'\times L'$, 
and $A'\times L'$ is its Zariski closure. 
Then 
$$
\delta_{\rm Zar}(U')=\dim A'+\dim L' -\dim U'.
$$
But $\dim U' \le \dim A'$ and $L\subset L'$. 
If 
$$
\delta_{\rm Zar}(U')\le \delta_{\rm Zar}(U)
$$
we must have $L=L'$ and $\dim U'=\dim A'$ so that
$U'$ is a pre-image of a ``dependent weakly special image''.
By the maximality of $U$ we have $U=U'$ and then $L=L'$
and $A=A'$ are weakly special.
\end{proof}

Now we get the finiteness statement. 

\begin{prop}\label{finitesp}
For each $k$ there are only finitely many strongly special subvarieties in $Y^k$ which have a 
$V$-image which lies in
any proper weakly special in $E^k$
\end{prop}

\begin{proof}
We take the definable space of products $M\times L$
of M\"obius and linear subvarieties, and take the
definable subset of maximal ones in the above sense.
These are Zariski optimal and hence geodesic optimal,
and hence are among the finite set of slopes
corresponding to the latter.
\end{proof}


\section{Arithmetic estimates}

Constant $C, C',\ldots, c, c',\ldots$ in the following depend on $E, Y, V, n$ and the choice
of a fundamental domain $F_Y$ for the uniformization $\Hh\rightarrow Y$.
We let $\Delta=\Delta(s)$ denote the discriminant
(which is negative) of a special point $s\in Y$.

\begin{prop}
Let $s\in Y$ be a special point and $\Delta(s)$ the discriminant 
of the corresponding quadratic order. Let $z\in F_Y$ be a pre-image
of $s$. Then


1. $h(s) \le c(\epsilon)|\Delta|^\epsilon$ for any $\epsilon>0$;

2. $H(z)\le C|\Delta(s)|^{c}$;

3. $[\Q(s):\Q]\le c(\epsilon)|\Delta|^{{1\over 2}+\epsilon}$ for any $\epsilon>0$;

4. $[\Q(s):\Q]\ge c(\epsilon)|\Delta|^{{1\over 2}-\epsilon}$ for any $\epsilon>0$.

\end{prop}

\begin{proof} For classical singular moduli:
1. Given in \cite{HABEGGERPILA}, Lemma 4.3. 2. Elementary (with $c=1$), given in \cite{PILA}.
3. See \cite{PAULIN} for an explicit result. 4. This is by the classical (ineffective)
Landau-Siegel bound.
The same bounds follow for a modular curve $Y$ as a finite 
cover of $Y(1)$. For Shimura curves: 2 follows from work of the second
author appearing in \cite{PTAS}, 1 follows from \cite{TSIMERMAN} combined with the comparison
(see e.g. \cite{PAZUKI}) of Faltings height with height of a moduli point,
while for 3 and 4 see \cite{ZHANG}, in particular equation (3.10) for $O_{\rm gl}(x)=O_{\rm cm}(x)$, 
and Remark (1) on Page 3664 for the asymptotic.
\end{proof}

We assume $E$ is in Weierstrass form (but an estimate of the same form
then follows if it isn't) and defined over a number field $K_0$
of degree $D=[K_0:\Q]$.
Let $q$ denote the N\'eron-Tate height on $E$ (see e.g. \cite{BRUIN} or \cite{MASSER}).

We have the following Theorem E of Masser \cite{MASSER}.
Set
$$
\eta=\eta(E, K)={\rm inf\ } q(x), 
$$
taking the infimum over non-torsion $x\in E(K)$, and let
$$
\omega=\omega(E,K)
$$
be the cardinality of the torsion subgroup of $E(K)$.

\begin{thm}
Let  $x_1,\ldots, x_n\in E(K)$ with 
N\'eron-Tate heights bounded by $q\ge \eta$.
There is a basis for the relations
$$
m_1x_1+\ldots+m_nx_n=0_E, \quad m_i\in \Z,
$$
with all $m_i$ having 
$$
|m_i|\le n^{n-1}\omega\Big({q\over\eta}\Big)^{(n-1)/2}. \qed
$$
\end{thm}

To accommodate CM, we work, like Barroero \cite{BARROERO}, in $E^{2n}$ with $x_i, \rho x_i$,
where $E$ has CM by the order $\Z+\Z\rho$. We write $||a+b\rho||=\max(|a|, |b|)$
for $a+b\rho\in {\rm End}(E)$. Then under the previous hypotheses
a set of generators for the relation group can be found with
$$
||m_i||\le (2n)^{2n-1}\omega\Big({q\over\eta}\Big)^{(2n-1)/2}.
$$

Following \cite{MASSER} we have the following estimates for $\eta, \omega$.
Set $L=\log(D+2)$. We have 
$$
\eta\ge C^{-1}D^{-3}L^{-2}
$$
by results of, respectively, Laurent (CM) and Masser (non CM) cited in \cite{MASSER},
and
$$
\omega\le CDL
$$
(see discussion in \cite{MASSER}).

Combining the above estimates yields the following result, where $||m||$
is as above in the CM case, but in the non-CM case we set $||m||=|m|$.

For a tuple $s=(s_1,\ldots, s_n)\in Y^n$ of special points with discriminants $\Delta(s_i)$
we define the {\it complexity\/} of $s$ by ${\bf \Delta}(s)=\max(|\Delta(s_i))|$.

\begin{prop}
There are constants $C, C', c$, depending on $E, Y, V, n$, with the following property.
Let $(s_1, x_1)\ldots, (s_n, x_n)\in Y\times E$ be $V$-graphs  of 
special points with discriminants $\Delta(s_i)$
and set ${\bf \Delta}={\bf \Delta}(s)={\bf \Delta}(s_1,\ldots, s_n)$. 
Then, for ${\bf \Delta}\ge C'$, there is a generating
set for the linear relations satisfied by the $x_i$ in $E$ with
$$
||m_i|| \le C{\bf \Delta}^c.
$$
\end{prop}

\begin{proof}

The difference $|q-h|$ is bounded on $E(\overline{K_0})$ by some constant $c^*$
(see e.g. \cite{BRUIN}).
On the other hand, if $x$ is a $V$-image of $s$ then
$H(x)\le CH(s)^c$ and $[K_0(x): K_0]\le C[\Q(s):\Q]$.
Thus, $D\le C{\bf \Delta}^c$ by 4.1.3.

If 
the maximum $h$ of the $h(x_i)$ 
is sufficiently large then we will have
$h-c^* \ge \eta$ and $2h\ge q$. Then $h\le C{\bf \Delta}^c$ by 4.1.1,
and now everything in 4.2 is bounded in terms of ${\bf\Delta}$.
\end{proof}

Propositions 4.3 and 4.1.2 will be used in the next section to bound the 
height of a rational/quadratic
point on a suitable definable set, while 4.1.4 will be used to show that there
are ``many'' such points.


\section{Proof of Theorems over $\overline{\Q}$\/}

\begin{proof}[Proof of Theorem 1.1 when $E, V$ are defined over $\overline{\Q}$]
Let $K_0$ be a numberfield over which $E, Y$, $V$ and all elements of 
${\rm End}(E)$ are defined.

We consider an exemplary special graph $W\subset V^n$, a $V$-image
of some special subvariety $S\subset Y^n$, with $\langle\pi_{E^n}(W)\rangle=B$.
Then any Galois conjugate $W'$ of $W$ over $K_0$
is also an exemplary special graph (of the conjugate $S'$ of $S$, with
$\langle\pi_{E^n}(W')\rangle=B'$ with $B'$ the corresponding conjugate of $B$), 
and {\it vice-versa}.

We can write $S$ as a product $S=S_1\times \{S_2\}$ of some strongly special $S_1\subset Y^{A_1}$
on some subset $A_1\subset \{1,\dots,n\}$ of coordinates, and a special 
point $S_2\in Y^{A_2}$ 
where $A_2\subset \{1,\dots,n\}$ is the complementary subset to $A_1$.

By Proposition 3.4 there are only finitely many such $S_1$ to consider, and so we may assume they are all defined over $K_0$.

We can write $W=W_1\times W_2$
and write $\xi_j, \eta_k$ for the coordinates in $E^{A_1}, E^{A_2}$
respectively.
We will show that if $\eta\in E^{A_2}$ is a $V^{A_2}$-image of a special point $S_2$ of 
sufficiently large complexity (depending on $S_1$)
then $W$ is not exemplary, and this will establish the requisite finiteness.

It may be that the projection of $W_1$ to $E^{A_1}$
is contained in some proper weakly special subvariety, which means that there are some
equations of the form
$$
\sum_{i\in A_1} m_i \xi_i=p,\quad m_i\in {\rm End}(E), \quad p\in E
$$
holding on this projection. We let $p_1,\dots,p_k$ be the points corresponding to a generating set
of such relations. Note that the linear span 
of the $p_i$ is ${\rm Gal}(\overline{\Q}/K_0)$ invariant, so we can make all the $p_i$ defined over 
$K_0$. 

If we take a generating set of all the equations over ${\rm End}(E)$ satisfied by the points in
$\pi_{E^n}(W)$ then this defines an algebraic subgroup $B_0$ of which $B$ is a connected
component. Any such equation of the form
$$
\sum_{i\in A_1} m_i\xi_i+\sum_{j\in A_2} n_j\eta_j=0,\quad m_i, n_j\in {\rm End}(E),
$$
entails that $\sum m_i\xi_i$ is constant on $W_1$ and 
is equivalent to some equation involving the $p_i, \eta_j$, and {\it vice-versa}.
We consider then the system of equations
$$
\sum_{i=1,\ldots, k} m'_i p_i+\sum_{j\in A_2} n_j\eta_j=0, \quad m'_i, n_j\in {\rm End}(E),
$$
corresponding (and equivalent) to the system defining $B_0$, where $\eta$
is a $V^{A_2}$-image of $S_2$. Let $d_0$ be the dimension of the subvariety this cuts
out in $E^{A_2}$.

By Proposition 4.3 there is a set of generators of all such relations with
$$
||m_i||, ||n_j|| \le C{\bf \Delta}(S_2)^c.
$$

Fix a pre-image $\nu=(\nu_1,\ldots, \nu_k)\in F_E^k$ of $(p_1,\ldots, p_k)$.
Let us first suppose that $E$ has NCM (``not CM''), and $d=\dim B$. 
Let $G$ be the Grassmanian of $(d_0+k)$-dimensional affine linear $\C$-subspaces of $\C^{k+n_2}$
where $n_2=|A_2|$.

Take the definable set
$$
X=\{(z, w, g)\in F_Y^{A_2} \times F_E^{A_2}\times G: u(z, w)\in V^{A_2}, (\nu,w)\in g\},
$$
where $F_E$ is a standard fundamental domain for the uniformization $\C\rightarrow E$,
and, projecting, the definable set
$$
Z=\{(z, g)\in F_Y^{A_2}\times G: \exists w\in F_E^{A_2}: (z, w, g)\in X\}.
$$

A special point $S_2\in Y^{A_2}$ of ``large'' complexity ${\bf \Delta}(s)$ leads to ``many'' points in $Z$
which are quadratic in the $F_Y$ coordinates and rational 
(even integral) in the $g$ coordinates.
More specifically, for sufficiently large ${\bf \Delta}(s)$ we get (by 4.1.4, 4.1.2, and 4.3)
$$
\gg {\bf \Delta}(S_2)^c {\rm\ such\ points\ of\ height\ at\ most\ } \ll {\bf \Delta}(S_2)^{c'}.
$$

Hence, by the Counting Theorem (see e.g. \cite{PW}),
there is a connected, semi-algebraic set $R$ in $Z$ belonging to a fixed definable family, in which
the $z$ coordinates cannot be constant (since the positive-dimensional semi-algebraic
sets need to account for ``many'' different conjugates of $s$). Since all of the Galois conjugates of a point have the same slopes $m_i,n_j$ we can
moreover assume that $R$ has a fixed slope.

\begin{lemma}\label{fixedcoset}

The projection of $R$ to $G$ is a point.

\end{lemma}

\begin{proof}

Let $\beta$ be the covering space of $B_0$ and $\beta' = \C^{A_2}/\beta$. 
Consider the image $R'\subset F_Y^{A_2}\times F_{\beta'}$ of the pre-image of $R$ in $X$. 
Again by the counting theorem, $R'$ 
contains a semi-algebraic set $R''$ belonging to a fixed definable family, with ``many" rational points coming from a single Galois orbit. Now note that $R''$ maps into the image $V'$ of $V^n$ 
inside the product $Y^{A_2}\times E^{A_2}/B_0$. Thus by Ax-Lindemann, the image of $R''$ lies in a weakly special contained in $V'$. However, the projection of $V'$ to $Y^{A_2}$ is finite-to-one, 
and therefore the weakly special  containing the image of $R''$ must have no abelian part, and therefore its projection to $E^{A_2}/B_0$ is a point, as desired.

\end{proof}

By lemma \ref{fixedcoset} we may write $R=A\times g_0$ with $g_0\in G$ and $A\subset\Hh^{A_2}$ semi-algebraic. Let $L$ be the linear subspace of $\C^{k+n_2}$ corresponding to 
$g_0$. Note that $L$ projects to some Galois conjugate of $B$ inside $E^{A_2}$. Let $L_\nu\subset \C^k$ be the fiber of $L$ over $\nu$. 
Now, by definition of $A$, we have that $A\times L_\nu \cap u^{-1}(V^{A_2})$ has a component $U$ which maps onto $A$. Note that the Zariski defect of $U$ is at most $d_0$.

By Proposition \ref{zardef}, there exists a weakly special $A^*$ containing $A$ and a component  $U^*$ of $A^*\times L_\nu \cap u^{-1}(V^{A_2})$ containing $U$ 
which maps onto $A^*$ with defect at most $d_0$. Since $A^*$ contains special points, it must in fact be special. 
Let $S^*$ be the image of $A^*$ in $Y^{A_2}$.
It contains at least one (in fact ``many'') Galois conjugates of $S_2$. By definition, a suitable $V$-image of $S_1\times S^*$ is contained in a coset of $B_0$.
We may now take a Galois conjugate of $S^*$ which contains $S_2$, thus giving a larger special graph projecting to the same torsion coset, which is a contradiction.


Now suppose that $E$ has CM by the order $\Z+\Z\rho$. 
We now let $G$ parameterize $(n_2+2k+d_0)$-dimensional complex affine-linear subspaces
in $\C^{2k+2n_2}$ and consider the definable set
$$
X=\{(z, w, g)\in F_Y^{A_2}\times F_E^{A_2}\times G: u(z, w)\in V^{A_2}, (\nu,\rho\nu, w, \rho w)\in g\}
$$
and, projecting, the definable set
$$
Z=\{(z, g)\in F_Y^{A_2}\times G: \exists w\in F_E^{A_2}: (z, w, g)\in X\}.
$$
The rest of the proof is the same as the NCM case.
\end{proof}

\begin{proof}[Proof of Theorem 1.5 when $E, V, U$ are defined over $\overline{\Q}$]
This is very much the same as the argument above but using different arithmetic estimates, drawn from \cite{HPA},
and a different definable set on which to count points.

We consider again an exemplary special graph of the form $W_1\times W_2$,
a $V$-image of some $S_1\times \{S_2\}$ as above with $S_2\in Y^{A_2}$ a special point.
There are again only finitely many such decompositions to consider, by 3.4.

Let us consider $U$-special points $S_2=(s_i)\in Y^{A_2}$ of a particular form, namely points in which
$s_i$ is in the Hecke orbit of a fixed $u_i\in U$ for $i\in A_2$, and all the $u_i$
are non-special.  Then there is a unique cyclic isogeny between the elliptic curves
corresponding to $u_i$ and $s_i$ whose degree we denote $N_i$.
For such a point $S_2$ we define its {\it $U$-complexity\/} by
$$
{\bf \Delta}(S_2)=\max\{N_1,\ldots, N_n\}.
$$

We observe that the height of $S_2$ is controlled by ${\bf \Delta}(S_2)$; using the results of Faltings relating
Faltings heights of isogenous elliptic curves and Silverman's comparison of Faltings height
and height of the $j$-invariant (see  the discussion in \cite{HPA}
on heights under isogenies in the proof of Lemma 4.2, p15)
we have
$$
h(S_2) \le C \max\{1, \log N_i\}
$$
(constants now depend on $Y, E, V, U$ and $n$).
If $(S_2, \eta)\in V^{A_2}$ the above leads (via Masser's Theorem E) to bounds of the form 
$$
||m||\le C{\bf \Delta}(S_2)^c
$$
on the size of entries in a set of generators for the relation group of  $(p,\eta)$.

On the other hand the degrees $[\Q(s_i):\Q]$ are controlled by ${\bf \Delta}(S_2)$ via isogeny estimates
(see the discussion in \cite{HPA} on degrees in \S6 above proof of 1.3) which imply
$[\Q(s_i):\Q] \ge C' N_i^{1/6}$ and hence
$$
[\Q(S_2):\Q] \ge C'{\bf \Delta}(S_2)^{c'}.
$$

Finally, if $\nu_i\in F_Y$ is a pre-image of $u_i$ and $z_i\in F_Y$ is a preimage of $s_i$
then $z_i=g\nu_i$ for some $g_i\in{\rm GL}_2^+(\Q)$ with 
$$
H(g_i)\le cN_i^{10}
$$ 
(see Lemma 5.2 of \cite{HPA}). 

We now count points though on a different definable set as $U$-special
points are not algebraic and the counting must be done for ${\rm GL}_2^+(\Q)$ points
in a definable subset of ${\rm GL}_2^+(\R)$.

We fix a pre-image $\mu\in F_Y^{A_2}$ of $(u_1,\ldots, u_{n_2})$ 
and consider the definable set
$$
X=\{(h, \nu, w, g)\in {\rm GL}_2^+(\R)^{A_2}\times F_E^k\times F_E^{A_2}\times G:
$$
$$
h\mu\in F_Y^{A_2}, u(h\mu, w)\in V^{A_2}, (\nu, w)\in g\}
$$
and its projection
$$
Z=\{(h,g)\in {\rm GL}_2^+(\R)^{A_2}\times G: \exists w\in F_E^{A_2}: (h\mu, \nu, w, g)\in X\}.
$$

A $U$-special point $S_2$ of the form being considered of ``large'' complexity leads to ``many'' rational points on $Z$. 
If ${\bf \Delta}(S_2)$ is sufficiently large then by counting we get a real algebraic curve in $Z$ which 
(since these come from ``many'' distinct points in  $F_Y^{A_2}$ and by complexification) gives 
rise to a complex algebraic curve $A\subset \Hh^{A_2}$ and an intersection component of $A\times L_g$ 
of Zariski defect $d$ as previously. This leads to a contradiction as in the argument above, so that ${\bf \Delta}(S_2)$ is bounded
for an exemplary special graph, giving finiteness for $S_2$ of this type.

The general case will follow by combining the treatment
of special and non-special points using a suitable definable set 
(i.e. using $F_Y$ for special coordinates and ${\rm GL}_2^+(\R)$
for coordinates in the Hecke orbit of a non-special $u\in U$) and a combinatorial argument.
\end{proof}

\section{Going from $\Qa$ to $\C$}

\subsection{Setup}

Let $F$ be a finitely generated subfield of $\C$ so that $V\subset Y(1)\times E$ are all defined over $F$. 
$F$ can be thought of as the function field of an irreducible algebraic variety $S$ over some number field $K\subset F$.
Replacing $S$ with a dense open subset, we assume that $E$ extends to an elliptic scheme $\mathcal{E}$ 
over $S$ and $V$ extends to a flat family $\mathcal{V}$ over $S$. We pick a generic regular point $s_0\in S(\C)$ 
such that $K(s_0)$ is isomorphic to $F$, and pick an open ball $B\subset S(\C)$ around $s_0$, 
so that in $B$ we can trivialize the homology of $\mathcal{E}$ over $S$.

\subsection{Ordering points in $S$}

We will need to order points in $S$, so we proceed as follows. Let $f:S\ra \mathbb{P}^{\dim S}$ be a quasi-finite map. 
Then we define the $f$-degree of a point $s$ in $S(\Qa)$ to be the degree of its image under $f$, 
and the $f$-height  $h_f(s)$ to be the (logarithmic) height of its image under $f$. By Northcott's theorem, 
there are finitely many points of bounded $f$-degree
and $f$-height. We only consider heights for the subset  $S_f$ of $S$ whose image lands in $\mathbb{P}^{\dim S}(K)$.

\subsection{The proof}

By Proposition \ref{finitesp} in a 
family, there are only finitely  many strongly special varieties whose $\mathcal{V}_u$-image lies inside 
any proper weakly-special subvariety of  $\mathcal{E}_u$ for any $u\in B$. Thus, there are only finitely many 
families of special subvarieties we have to consider. By rearranging co-ordinates, we may assume they are 
all of the form $T\times p\times q$ where $T\subset Y(1)^m$ is a fixed strongly special subvariety, and $p\in Y(1)^k$  is 
is a CM point, and $q$ has co-ordinates isogenous to points in $U$.

Now, for the sake of contradiction let $p_i,q_{i,s_0}$ be an infinite sequence of such points 
such that $T\times p_i\times q_{i,s_0}$ 
are projections of optimal special graphs for $\mathcal{V}_{s_0}$. Let $A_i$ be the smallest torsion coset
containing the $\mathcal{V}_{s_0}$-image of $T\times p_i\times q_{i,s_0}$. Then for each point $s\in S(\Qa)\cap B$ 
image of $T\times p_i\times q_{i,s}$ is still contained in $A_i$. But we've proven the statement for $\Qa$-points, 
and thus for each $s$ there are finitely many special varieties containing all the $T\times p_i\times q_{i,s}$ whose 
$\mathcal{V}_s$ image is contained in a proper torsion coset.
 
Let $T_1(s_0),\dots,T_m(s_0)$ be the smallest collection of $U_{s_0}$-special subvarieties 
containing all the $T\times p_i\times q_{i,s_0}$.

\begin{lemma}\label{genericspecialclosure}

For large enough $d$, for a density 1 set of points $s$ in $S_f$ ordered by $f$-height, 
$T_1(s),\dots,T_m(s)$ is the smallest collection of $U_s$-special subvarieties containing all the $T\times p_i\times q_{i,s}$. 

\end{lemma}

\begin{proof}

First, note that since the degrees of CM points tend to infinity. Thus, the set of points $s\in S_f$ 
such that $U_s$ is CM is contained in a proper subvariety, and so has density 0.
Next, since $U$-special subvarieties are defined simply by imposing isogeny relations, it is sufficient to 
prove that for a density 1 set of points $s$ that $u_s,v_s$ are not isogenous,
for $u,v$ distinct points in $U$. 

Now, for $s\in S_f$, it follows that $h(u_s),h(v_s)\ll h_f(s)$, and thus by Masser-W\"ustholz isogeny bound 
\cite[Main Theorem]{MW} it follows that if $u_s,v_s$ are isogenous then there
is an isogeny between them of degree $O(h_f(s)^\kappa)$ for some fixed $\kappa>0$. 
Now, the degree of the $T_N$ in $X(1)^2$ is $O(N^2)$, and therefore the set of all $s\in S_f$ 
with $h_f(s)< X$ such that $u_s,v_s$ are isogenous are
contained in $O(h_f(s)^{\kappa})$ divisors of $f$-degree at most $O(h_f(s)^{2\kappa})$. 
Now, the size of $\{s\in S_f, h(s)<X\}$ is asymptotic to $e^{X(\dim S+1)}$ whereas the number
of points in any divisor of degree $d$ of height at most $X$ is $O(de^{X\dim S})$. The result follows.

\end{proof}
 

Thus we are done once we prove the following
 
\begin{lemma}

 Let $\mathcal{E}$ be an elliptic scheme over $S$, and let $\mathcal{W}\subset\mathcal{E}^n$ be an irreducible 
 algebraic subvariety. If $\mathcal{W}_s$ is contained inside a proper abelian subvariety for a density 1 set of $s\in S_f$,
 then $\mathcal{W}$ is contained inside an abelian subscheme.
 
\end{lemma}
 
\begin{proof}
 
Replacing $\mathcal{W}$ by its own $n$-fold self-sum we may assume that $\mathcal{W}$ is a coset of an 
abelian subscheme. Quotienting out by the corresponding abelian subscheme, we may further assume 
that $\mathcal{W}$ is finite over $S$, and base changing $S$ by a finite map we may assume that $\mathcal{W}$ is a section over $S$. 
By the Main Theorem of \cite{MASSER2}, it follows that for a density one set of points $s$ the $n$ points of $\mathcal{E}_s$ 
represented by $\mathcal{W}_s$ are linearly independent. This completes the proof in the case that $\mathcal{E}$ 
does not have generic CM. 
Otherwise, one may argue similarly, by recording an extra set of co-ordinates 
for the extra endomorphism of $\mathcal{E}$.
\end{proof}
  
\noindent
{\bf Acknowledgements.\/} JP thanks EPSRC for support under grant reference
EP$\backslash$N008359$\backslash$1. JT thanks NSERC and the Ontario Early Researcher Award for support.

\bigskip
\bigskip

\noindent
\leftline{JP: Mathematical Institute, 
University of Oxford, Oxford, UK.}
\rightline{pila@maths.ox.ac.uk}

\bigskip
\bigskip
\noindent
\leftline{JT: Department of Mathematics, University of Toronto, Toronto, Canada.}
\rightline{jacobt@math.toronto.edu}

\vfil
\eject

\end{document}